\documentclass[11pt]{article}
\usepackage{amssymb,amsmath,amsthm,graphicx}
\usepackage{color,enumerate}
\usepackage{authblk}
\usepackage{cite,fullpage}
\usepackage{caption}
\usepackage{subcaption}
\usepackage[ruled,vlined,linesnumbered]{algorithm2e}
\usepackage{relsize}
\usepackage{rotating}
\usepackage{url}

\newtheorem{theorem}{Theorem}[section]
\newtheorem{lemma}[theorem]{Lemma}
\newtheorem{corollary}[theorem]{Corollary}
\newtheorem{proposition}[theorem]{Proposition}
\newtheorem{question}[theorem]{Question}

\newtheorem{problem}[theorem]{Problem}

\newtheorem{example}[theorem]{Example}

\theoremstyle{definition}
\newtheorem{definition}[theorem]{Definition}

\newcommand{\Aut}{\hbox{{\rm Aut}}}
\newcommand{\ZZ}{\mathbb{Z}}

\usepackage{amssymb}
\usepackage{graphicx}
\begin{document}

\title{Strong cliques in vertex-transitive graphs}

\author[1,2]{Ademir Hujdurovi\'c\thanks{ademir.hujdurovic@upr.si}}
\affil[1]{\normalsize University of Primorska, UP IAM, Muzejski trg 2, SI-6000 Koper, Slovenia}
\affil[2]{\normalsize University of Primorska, UP FAMNIT, Glagolja\v ska 8, SI-6000 Koper, Slovenia}

\date{\today}

\maketitle
\begin{abstract}
A clique (resp., independent set) in a graph is strong if it intersects every maximal independent sets (resp., every maximal cliques). A graph is CIS if all of its maximal cliques are strong and localizable if it admits a partition of its vertex set into strong cliques. In this paper we prove that a clique $C$ in a vertex-transitive graph $\Gamma$ is strong if and only if $|C||I|=|V(\Gamma)|$ for every maximal independent set $I$ of $\Gamma$. Based on this result we prove that a vertex-transitive graph is CIS if and only if it admits a strong clique and a strong independent set. 
We classify all vertex-transitive graphs of valency at most 4 admitting a strong clique, and give a partial characterization of $5$-valent vertex-transitive graphs admitting a strong clique. Our results imply that every vertex-transitive graph of valency at most $5$ that admits a strong clique is localizable.
We answer an open question by providing an example of a vertex-transitive CIS graph which is not localizable. 
\end{abstract}


\maketitle

\section{Introduction}

A {\em clique} (resp., {\em independent set}) in a graph is a set of pairwise adjacent (resp., pairwise non-adjacent) vertices.
A clique (resp.,~independent set) in a graph is said to be {\em maximal} if it is not contained in any larger clique (resp., independent set),
and it is said to be {\em maximum}, if it is of the maximal size.
{\em Clique} number of a graph $\Gamma$, denoted by $\omega(\Gamma)$ is the size of a maximum clique, and {\em independence number}, denoted by $\alpha(\Gamma)$ is the size of a maximum independent set.
A { clique} (resp., { independent set}) is said to be {\em strong} if it intersects every maximal independent set (resp.,~every maximal clique).
It is co-NP-complete to test whether a given clique in a graph is strong~\cite{MR1344757} and NP-hard to test whether a given graph contains a strong clique~\cite{MR1301855}. 

The notions of strong cliques and strong independent sets in graphs played an
important role in the study of several graph classes. For example, {\it strongly perfect graphs} \cite{MR715895,MR778749} are graphs in which each of its induced subgraphs has a strong independent set, and {\it very strongly perfect graphs} (also called {\it Meyniel graphs}) \cite{MR888682,MR0439682,MR778765} 
are graphs such that in each of its induced subgraphs each vertex belongs to a strong independent set. Another class of graphs defined using strong cliques are {\em localizable graphs}, defined as graphs whose vertex-set can be partitioned into strong cliques. They were introduced by Yamashita and Kameda in 1999, in the concluding remarks of their paper~\cite{MR1715546}, and were studied further in \cite{HujdurovicMR2018}. Localizable graphs form a subclass of  {\em well-covered} graphs~\cite{MR1254158,MR1677797}, which are defined as graphs in which all maximal independent sets have the same size.
A graph is said to be {\it CIS} if its every maximal clique is strong, or, equivalently, if its every maximal
independent set is strong, see \cite{Boros2015,MR3141630,MR2496915,MR2489416,MR2755907}.

The study of strong cliques in the context of vertex-transitive graphs was initiated in \cite{MR3278773}, where it is proved that a vertex-transitive graph $\Gamma$ is CIS if and only if it is well-covered, its complement is well covered and $\alpha(\Gamma)\omega(\Gamma)=|V(\Gamma)|$, generalizing analogous result about CIS circulants proved in \cite{MR3141630}. In addition, classification of vertex-transitive CIS graphs with clique number at most $3$ (see \cite[Theorem 4.3]{MR3278773}) and classification of vertex-transitive CIS graphs of valency at most $7$ (see \cite[Corollary 5.6]{MR3278773}) was given.

In this paper we continue with the study of strong cliques in vertex-transitive graphs initiated in \cite{MR3278773}. We 
prove that a clique $C$ in a vertex-transitive graph $\Gamma$ is strong if and  only if $|C||I|=|V(\Gamma)|$ for every maximal independent set $I$ of $\Gamma$
(see Theorem~\ref{thm:strongclique}), and based on this result we prove that a vertex-transitive graph is CIS if and only if it admits a strong clique and a strong independent set (see Proposition~\ref{prop:CIS}). We also observe that existence of a strong clique in a vertex-transitive graph implies that the graph is well-covered,  and that every strong clique in a vertex-transitive graph must be maximum (see Corollary~\ref{cor:strongclique}). Furthermore, in Section~\ref{sec:CIS non localizable} we give a negative answer to \cite[Question 6.2]{MR3278773} by constructing a vertex-transitive CIS graph which is not localizable.  

In Section~\ref{sec:Small valency} we investigate the existence of strong cliques in vertex-transitive graphs of small valency. We prove that there are only 3 cubic vertex-transitive graphs admitting a strong clique (see Theorem~\ref{thm:VT-3valent}), and classify all $4$-valent vertex-transitive graphs admitting a strong clique (see Theorem~\ref{thm:4valent-main}). For the $5$-valent case, we give a partial characterization. Namely, we prove that there is a unique $5$-valent vertex-transitive graph admitting a strong clique with a given clique number different from $4$ (see Proposition~\ref{prop:5 valent omega=2,3,5,6}). For the $5$-valent vertex-transitive graphs with clique number $4$ admitting a strong clique, we classify them in Proposition~\ref{prop:5-valent local graph not L1}, except in the case when the subgraph induced by the neighbours of a given vertex consist of a disjoint union of $K_3$ and two isolated vertices, in which case we provide $4$-infinite families of examples (see Example~\ref{ex:local graph L1}). Our results imply that every vertex-transitive graph of valency at most $5$ admitting a strong clique is localizable. We conclude the paper by posing several open problems in Section~\ref{sec:problems}.


\section{Preliminaries}

All graphs and groups considered in this paper are finite.
Let $\Gamma= (V, E)$ be a graph. We call $V$ the
vertex set of $\Gamma$ and write $V = V(\Gamma )$. Similarly, we call $E$ the edge set of $\Gamma$ and write $E = E(\Gamma)$.
The {\em complement}  $\overline{\Gamma}$ of $\Gamma$ is the graph with the same vertex set and the complementary edge set.
For a vertex $v \in V(\Gamma)$, let $\Gamma(v)$ denote the neighborhood of $v$, that is, the set of vertices of $\Gamma$ that are
adjacent to $v$. If $v$ is a non-isolated vertex of
$\Gamma$ then the {\em local graph} of $\Gamma$ at $v$ is the subgraph of $\Gamma$ induced by $\Gamma(v)$. We sometimes denote the local graph of $\Gamma$ at $v$ with $\Gamma(v)$. Whether with $\Gamma(v)$ we mean the set of neighbours of $v$ in $\Gamma$ or the local graph at $v$ should always be clear from the context. With $C_n$ we denote the cycle of length $n$ and with $K_n$ we denote the complete graph of order $n$.
An {\it automorphism} of a graph $\Gamma=(V,E)$ is a bijective mapping $\varphi: V\to V$ such that
$$
(\forall  u,v\in V)\quad \{u,v\}\in E\Leftrightarrow \{\varphi(u),\varphi(v)\}\in E.
$$
The set of all automorphisms of a graph $\Gamma$ is denoted with $Aut(\Gamma)$, and is called the {\it automorphism group} of $\Gamma$.
A graph $\Gamma$ is said to be {\it vertex-transitive} if for any two vertices $u$ and $v$ of $\Gamma$ there exists an automorphism $\varphi$ of $\Gamma$ such that $\varphi(u)=v$. An important subclass of the family of vertex-transitive graphs is formed by {\em Cayley graphs}, defined as follows. For a group $G$ and an inverse closed subset $S\subseteq G\setminus \{1_G\}$, Cayley graph of $G$ with respect to the connection set $S$, denoted by $Cay(G,S)$, is the graph with vertex set $G$, and two vertices $x,y\in G$ being adjacent if and only if $x^{-1}y\in S$.

A graph is said to be {\em reducible} if it has two distinct vertices with the same neighbourhoods, and {\em irreducible} otherwise.
The {\it lexicographic product} of graphs $\Gamma_1$ and $\Gamma_2$ is the graph $\Gamma_1[\Gamma_2]$ with vertex set $V(\Gamma_1)\times V(\Gamma_2)$, where two vertices $(u,x)$ and $(v,y)$ are adjacent if and only if either $\{u,v\}\in E(\Gamma_1)$ or $u = v$ and $\{x,y\}\in E(\Gamma_2)$.
A vertex-transitive graph $\Gamma$ which is reducible is isomorphic to the lexicographic product $\Gamma'[nK_1]$, where $\Gamma'$ is irreducible vertex-transitive graph (see \cite[Proposition 5.5]{MR3278773}).

Since automorphisms preserve adjacency relations, it is easy to see that any  automorphism of a graph maps cliques, maximal cliques and strong cliques into cliques, maximal cliques and strong cliques, respectively.

Let $\Gamma$ be a graph and  $A,B\subseteq V(\Gamma)$. We say that $A$ {\em dominates} $B$ if every vertex of $B$ has a neighbour in $A$.
The following lemma gives a simple characterization of strong cliques based on domination. 
\begin{lemma}\cite[Lemma 2.1]{HMR-new}\label{lem:dominates}
Let $\Gamma$ be a graph and $C$ a clique in $\Gamma$. Then $C$ is a strong clique if and only if no independent set disjoint with $C$ dominates $C$.
\end{lemma}
%

\begin{proposition}\cite[Lemma 6.7]{HMR-new}\label{prop:trianglefree-regular}
Let $\Gamma$ be a connected $m$-regular graph. Then $\Gamma$ has a strong clique of size $2$ if and only if $\Gamma\cong K_{m,m}$. 
\end{proposition}

\section{Characterization of strong cliques in vertex-transitive graphs}

In this section we will give a characterization of strong cliques in vertex-transitive graphs. The main result is stated in the following theorem.

\begin{theorem}\label{thm:strongclique}
Let $\Gamma$ be a vertex-transitive graph and let $C$ be a clique of $\Gamma$. Then $C$ is a strong clique if and only if  $|C||I|=|V(\Gamma)|$ for every maximal independent set $I$ of $\Gamma$.
\end{theorem}
\begin{proof}
Let $C$ be a given clique of $\Gamma$ and denote with 
 $G$ be the automorphism group of $\Gamma$. Let $I$ be an arbitrary maximal independent set of $\Gamma$. Denote by $\Omega(I)$ the set of triples $(x,y,\varphi)$ such that $x\in C$, $y\in I$, $\varphi\in G$ and $\varphi(x)=y$. Given $(x,y)\in C\times I$, the set of elements of $G$ which map $x$ to $y$ is a coset of $G_x$ and hence has cardinality $|G_x|$ (see \cite{MR1829620}[Lemma 2.2.1]). This shows that 
\begin{equation}\label{eq:omega}
|\Omega(I)|=|C||I||G_x|.
\end{equation}

Let $\varphi\in G$ be an arbitrary element of $G$. Then $(x,y,\varphi)\in \Omega(I)$ if and only if $y\in \varphi(C)\cap I$ and $x=\varphi^{-1}(y)$. Therefore
$$
|\{(x,y,\varphi)\mid x\in C,~ y \in I,~\varphi(x)=y \}|=|\varphi(C)\cap I|\leq 1.
$$

This shows, that for any $\varphi \in G$, the number of elements of $\Omega(I)$ with third coordinate equal to $\varphi$ is $|\varphi(C)\cap I|$. Therefore, $|\Omega(I)|\leq |G|$ with equality if and only if $\varphi(C)\cap I\neq \emptyset$, for every $\varphi\in G$. 

Suppose first that $C$ is a strong clique.  Then $\varphi(C)$ is a strong clique for every $\varphi\in G$ and hence  $|\varphi(C)\cap I|=1$, for every maximal independent set $I$ and every $\varphi\in G$. Therefore $|\Omega(I)|=|G|$.  By the orbit-stabilizer theorem (see \cite[Lemma 2.2.2]{MR1829620}), we have $|G|=|V(\Gamma)||G_x|$ and combining this with \eqref{eq:omega} we obtain $|V(\Gamma)|=|C||I|$, for every maximal independent set $I$.

Conversely, suppose that  $|C||I|=|V(\Gamma)|$, for every maximal independent set $I$. Then $|\Omega(I)|=|C||I||G_x|=|V(\Gamma)||G_x|=|G|$.
As noted above, we have that $|\Omega(I)|=|G|$, if and only if $\varphi(C)\cap I\neq \emptyset$, for every $\varphi\in G$. Since $|\Omega(I)|=|G|$, for every maximal independent set $I$, it follows that $C$ is a strong clique.
\end{proof}

We now state several corollaries of Theorem~\ref{thm:strongclique}.

\begin{corollary}\label{cor:strongclique}
Let $\Gamma$ be a vertex-transitive graph. If there exists a strong clique in $\Gamma$, then $\Gamma$ is well-covered. Moreover, every strong clique in $\Gamma$ is maximum.
\end{corollary}
\begin{proof}
Suppose that $C$ is a strong clique in a vertex-transitive graph $\Gamma$. Then $|C||I|=|V(\Gamma)|$, for every maximal independent set $I$. Therefore, all maximal independent sets in $\Gamma$ are of the same size. Moreover, since $\alpha(\Gamma)\omega(\Gamma)\leq |V(\Gamma)|$ (see, e.g.,~\cite[Chapter 7, Exercise 8]{MR1829620}), we infer that $C$ is a maximum clique.
\end{proof}

{\em Chromatic number} of a graph $\Gamma$, denoted by $\chi(\Gamma)$ is the minimum number of independent sets that partition $V(\Gamma)$. It is easy to see that for every graph $\Gamma$ we have $\omega(\Gamma)\leq \chi(\Gamma)$.
For some results regarding chromatic number of vertex-transitive graphs (and some of the subclases of vertex-transitive graphs) see \cite{Cranston,Konstantinova,Godsil,Klotz,Alon,Mrazovic,Green,Ilic}.

\begin{corollary}
Let $\Gamma$ be a vertex-transitive graph admitting a strong clique. Then $\Gamma$ is localizable if and only if $\chi (\overline{\Gamma})=\omega(\overline{\Gamma})$.
\end{corollary}
\begin{proof}
Let $\Gamma$ be a vertex-transitive graph admitting a strong clique. Then by Corollary~\ref{cor:strongclique} it follows that $\Gamma$ is well-covered. 
By \cite[Theorem 2.1]{HujdurovicMR2018}(d) it follows that $\Gamma$ is localizable if and only if $\alpha(\Gamma)=\theta(\Gamma)$, where $\theta$ denotes the clique-cover number, that is, it is the minimal number of cliques that cover the vertex-set. Since $\theta(\Gamma)=\chi(\overline{\Gamma})$ and $\alpha(\Gamma)=\omega(\overline{\Gamma})$ the result follows.
\end{proof}

\begin{corollary}
If a vertex-transitive graph $\Gamma$ has a maximal clique of size $|V(\Gamma)|/2$, then $\Gamma$ is localizable.
\end{corollary}
\begin{proof}
Let $C$ be a maximal clique in $\Gamma$ of size $|V(\Gamma)|/2$. Then, since $\alpha(\Gamma)\omega(\Gamma)\leq |V(\Gamma)|$, it follows that $\alpha(\Gamma)\leq 2$. Since $\Gamma$ is not a complete graph, it follows that there is no maximal independent set of $\Gamma$ of size 1. Hence all maximal-independent sets of $\Gamma$ are of size $2$. Then by Theorem~\ref{thm:strongclique} it follows that $C$ is a strong clique. Let $C_1=V(\Gamma)\setminus C$. We claim that $C_1$ is also a clique. Suppose that $C_1$ is not a clique, and let $x,y\in C_1$ be non-adjacent vertices. Then $\{x,y\}$ is an independent set in $\Gamma$, and by the above observation, it must be maximal. However, this contradicts the previously proved fact that $C$ is strong clique, since $C\cap \{x,y\}=\emptyset$. 
The obtained contradiction shows that $C_1$ is a clique, and since it is of size $|V(\Gamma)|/2$, it follows by Theorem~\ref{thm:strongclique} that it is strong. Hence, $\{C,C_1\}$ is a partition of $V(\Gamma)$ into strong cliques.
\end{proof}

The following example shows that the previous result cannot be generalized further, namely, for every $n\geq 3$ there exists a vertex-transitive graph admitting a strong clique of size $|V(\Gamma)|/n$ which is not localizable.

\begin{example}
For $n\geq 3$, $L(K_{2n})$ is vertex-transitive, has strong cliques and $\omega(L(K_{2n}))= |V(L(K_{2n}))|/n$. Moreover, this graph is not localizable.
\end{example}

In the following proposition, we prove that for testing if a given vertex-transitive graph is CIS, it suffices to check if a given maximal clique and a given maximal independent set are strong.
\begin{proposition}\label{prop:CIS}
A vertex-transitive graph is CIS if and only if it admits a strong clique and a strong independent set.
\end{proposition}
\begin{proof}
If $\Gamma$ is CIS, then every maximal clique and every maximal independent set is strong.
Suppose that $\Gamma$ is a vertex-transitive graph and suppose that $C$ is a strong clique and $I$ is a strong independent set in $\Gamma$. Then $I$ is a strong clique in $\overline{\Gamma}$. By Theorem~\ref{thm:strongclique} it follows that $|C||I'|=|V(\Gamma)$, for every maximal independent set $I'$ in $\Gamma$ and $|C'||I|=|V(\Gamma)|$, for every maximal clique $C'$ in $\Gamma$. This implies that all maximal cliques in $\Gamma$ are of the same size, all maximal independent sets in $\Gamma$ are of the same size and $\alpha(\Gamma)\omega(\Gamma)=|V(\Gamma)|$. By \cite[Theorem 3.1]{MR3278773} it follows that $\Gamma$ is CIS.
\end{proof}

We conclude this section with the following lemma.
\begin{lemma}\label{lem:irreducible}
Let $\Gamma$ be an irreducible vertex-transitive graph and let $C$ be a strong clique in $\Gamma$. Then for any maximal clique $C'\neq C$ of $\Gamma$ we have $|C\cap C'|<|C|-1$.
\end{lemma}
\begin{proof}
Let $C$ be a strong clique in an irreducible vertex-transitive graph $\Gamma$. Then by Corollary~\ref{cor:strongclique} it follows that $|C|$ is a maximum clique in $\Gamma$. Let $C'$ be a maximal clique in $\Gamma$ different from $C$. Suppose that $|C\cap C'|\geq |C|-1$. It is clear that $|C\cap C'|\neq |C|$ since $C$ is maximal clique, and $C'\neq C$. Therefore, $|C\cap C'|= |C|-1$. Since $C'$ is a maximal clique, and it is different from $C$ it follows that $|C'|=|C|$.
Let $C\setminus C'=\{v_1\}$ and $C'\setminus C=\{v_2\}$.
Let $x$ be a neighbour of $v_1$. If $x$ is not adjacent to $v_2$, then $\{x,v_2\}$ is an independent set disjoint with $C$ which dominates $C$. By Lemma~\ref{lem:dominates} this is contradiction with the assumption that $C$ is strong clique in $\Gamma$. Therefore, $x$ is adjacent to $v_2$. This shows that $\Gamma(v_1)=\Gamma(v_2)$, contrary to the assumption that $\Gamma$ is irreducible. 
\end{proof}

\section{Vertex-transitive CIS non localizable graphs}\label{sec:CIS non localizable}

In \cite[Question 6.2]{MR3278773} the following question was posed.

\begin{question}\cite{MR3278773}
\label{que:6.2}
Does every vertex-transitive CIS graph $\Gamma$ admit a decomposition of its vertex set into $\omega(\Gamma)$ independent sets?
\end{question}
The main result of this section is to answer the above question.
It is easy to see that Question~\ref{que:6.2} is equivalent to asking if every vertex-transitive CIS graph is localizable.
Namely, if $\Gamma$ is a vertex-transitive CIS graph, then $\overline{\Gamma}$ is also vertex-transitive CIS graph and it is easy to see that $\Gamma$ admits a decomposition of its vertex set into $\omega(\Gamma)$ independent sets if and only if $\overline{\Gamma}$ admits a decomposition of its vertex set into $\alpha(\Gamma)$ cliques.


Before answering Question~\ref{que:6.2} we need one more definition and a lemma.
For a graph $\Gamma$ we denote with $\Gamma_{\mathcal{Q}}$ the {\it graph of maximal cliques of $\Gamma$}, that is the graph with maximal cliques of $\Gamma$ as vertices, and two such vertices adjacent in $\Gamma_{\mathcal{Q}}$ if and only if they have non-empty intersection.

\begin{lemma}\label{lem:cliquegraph}
Let $\Gamma$ be a  vertex-transitive CIS graph. Then $\Gamma$ is localizable if and only if $\alpha(\Gamma_{\mathcal{Q}})=\alpha(\Gamma)$.
\end{lemma}
\begin{proof}
Let $\Gamma$ be a  vertex-transitive CIS graph. Then $\Gamma$ is well-covered, co-well-covered, and $|V(\Gamma)|=\alpha(\Gamma)\cdot \omega(\Gamma)$. Suppose first that $\Gamma$ is localizable. Then there exists a partition of $V(\Gamma)$ into $\alpha(\Gamma)$ maximal cliques. This corresponds to a maximal independent set in $\Gamma_{\mathcal{Q}}$ of size $\alpha(\Gamma)$. Hence $\alpha(\Gamma_{\mathcal{Q}})\geq \alpha(\Gamma)$. Moreover, since $\Gamma$ is co-well-covered, and $|V(\Gamma)|=\alpha(\Gamma)\cdot \omega(\Gamma)$  it follows that there can't exist more than $\alpha(\Gamma)$ pairwise disjoint maximal cliques of $\Gamma$. Therefore $\alpha(\Gamma_{\mathcal{Q}})= \alpha(\Gamma)$.

Suppose now that $\alpha(\Gamma_{\mathcal{Q}})=\alpha(\Gamma)$. Then there exists an independent set in $\Gamma_{\mathcal{Q}}$ of size $\alpha(\Gamma)$, or equivalently, there exists $\alpha(\Gamma)$ pairwise disjoint maximal cliques of $\Gamma$. Since $\Gamma$ is co-well-covered, and $|V(\Gamma)|=\alpha(\Gamma)\cdot \omega(\Gamma)$ it follows that these $\alpha(\Gamma)$ maximal cliques form a partition of $V(\Gamma)$. Hence $\Gamma$ is localizable.
\end{proof}

Let $n>k>i\geq 0$ be integers, and let $I_n=\{1,\ldots,n\}$. {\it Generalized Johnson graph} $J(n,k,i)$ has as vertex set all $k$-elements subsets of $I_n$, with two vertices being adjacent if their intersection is of size $k$. For simplicity, vertex $\{x,y,\ldots,z\}$ of generalized Johnson graph will be written as $xy\ldots z$.

\begin{theorem}\label{thm:J(7,3,1) is CIS not localizable}
Generalized Johnson graph $J(7,3,1)$ is vertex-transitive CIS graph which is not localizable.
\end{theorem}
\begin{proof}
It is clear that $\Gamma=J(7,3,1)$ is vertex-transitive, since $S_7$ acts transitively on its vertex-set.
Let $I$ be a maximal independent set in $\Gamma$. Suppose first that there exist two elements of $I$ which are disjoint. Without loss of generality, we may assume that these two vertices are $123$ and $456$. 
Let $xyz$ be another element of $I$. It is clear that $7\in \{x,y,z\}$, since otherwise, $xyz$ would be adjacent to one of $123$ and $456$. Let $z=7$. Since $xyz$ is not adjacent to $123$, it follows that $\{x,y\}\subset \{1,2,3\}$ or $\{x,y\}\cap  \{1,2,3\} =\emptyset$. It is now easy to see that $I$ has to be one of $\{123,456,127,137,237\}$ or $\{123,456,457,467,567\}$.

Suppose now that no two elements of $I$ are disjoint. Let $123$ and $124$ be two elements of $I$. It is not difficult to verify that the only maximal independent set satisfying this assumption is $I=\{123,124,125,126,127\}$. This shows that every maximal independent set in $\Gamma$ is of size $5$. 

Let $C$ be a maximal clique in $\Gamma$. Without loss of generality, we may assume that $123,145 \in C$. We claim that $167\in C$. Let $xyz\in C$  with $1\not \in \{x,y,z\}$. It follows that $|\{x,y,z\} \cap \{2,3\}|=1$ and $| \{x,y,z\} \cap \{4,5\} |=1$, which implies that also  $| \{x,y,z\} \cap \{6,7\} |=1$, which shows that $xyz$ is adjacent to $167$. This shows that every vertex of $C$ different from $167$ is adjacent to $167$, and since $C$ is maximal clique, we have $167\in C$.

Observe now that each of the 8 vertices
$xyz$ with $x\in \{2,3\}$, $y\in \{4,5\}$, $z \in \{6,7\}$ is adjacent with each of $123$, $145$, and $167$. 
It is now easy to see that there are two possibilities for $C$, namely $C_1=\{123,145,167, 246, 257, 347, 356\}$ and $C_2=\{123,145,167, 247, 256, 346, 357\}$. This shows that every maximal clique of $\Gamma$ is of size $7$ and moreover, each edge belongs to exactly $2$ maximal cliques. By \cite[Theorem 3.1]{MR3278773} it follows that $\Gamma$ is CIS graph. 

Since each edge of $\Gamma$ belongs to two maximal cliques, it follows that there are $30$ maximal cliques in $\Gamma$.
Observe that permutation $(6,7)\in S_7$ interchanges the two cliques $C_1$ and $C_2$. Combining this with the fact that $S_7$ acts transitively on the edge set of $\Gamma$ implies that $S_7$ acts transitively on the set of maximal cliques of $\Gamma$. Therefore, the clique graph $\Gamma_{\mathcal{Q}} $ of $\Gamma$ is vertex-transitive of order $30$.

Let $g=(1~2~3~4~5~6~7)\in S_7$ and let $\mathcal{C}={\{C_1^{(g^i)}:i=0,\ldots,6\}}$, that is $\mathcal{C}$ is an orbit of $\langle g \rangle$ containing $C_1$. Since $\langle g \rangle$ is of order $7$, its orbit can have size $1$ or $7$. Moreover, since $C_1^g\neq C_1$, it follows that  $\mathcal{C}$ is of size $7$.
We claim that any two elements of $\mathcal{C}$ have non-empty intersection. Observe that $347^g=145$, $167^{g^2}=123$, $246^{g^3}=572$, $257^{g^4}=246$, $123^{g^5}=167$ and $145^{g^6}=347$. This shows that $C_1$ has non-empty intersection with each element of $\mathcal{C}$, and since $\mathcal{C}$ is $\langle g \rangle$ orbit of $C_1$, it follows that any two elements of $\mathcal{C}$ have non-empty intersection. This shows that $\mathcal{C}$ is a clique of size $7$ in $\Gamma_{\mathcal{Q}}$. Therefore, $\omega(\Gamma_{\mathcal{Q}})\geq 7$. Since $\Gamma_{\mathcal{Q}}$ is vertex-transitive, it follows that $\alpha(\Gamma_{\mathcal{Q}})\cdot \omega(\Gamma_{\mathcal{Q}})\leq 30$, and consequently, $\alpha(\Gamma_{\mathcal{Q}})\leq 4$. Since $\alpha(\Gamma)=5$, by Lemma~\ref{lem:cliquegraph} it follows that $\Gamma$ is not localizable. 
\end{proof}
The following corollary gives a negative answer to Question~\ref{que:6.2}.
\begin{corollary}
The complement $\Gamma$ of generalized Johnson graph $J(7,3,1)$ is vertex-transitive CIS graph that does not admit a decomposition of its vertex set into $\omega(\Gamma)$ independent sets. 
\end{corollary}

\section{Small valent vertex-transitive graphs admitting a strong clique}
\label{sec:Small valency}

In this section we study existence of strong cliques in vertex-transitive graphs of small valency. 
The only connected graphs of valency 2 are cycles, and by Proposition~\ref{prop:trianglefree-regular} it follows that the only cycles that admit a strong clique are $C_3$ and $C_4$. The next step is to study existence of strong cliques in cubic vertex-transitive graphs. 
In \cite[Theorem 6.8]{HMR-new} it is proved that the only cubic graphs in which each vertex belongs to a strong clique are $K_4$, $K_{3,3}$, $\overline{C_6}$ and an infinite family of graphs denoted by $F_n$ in \cite{HMR-new}. Classification of cubic vertex-transitive graphs admitting a strong clique is now easily derived from \cite[Theorem 6.8]{HMR-new}. Namely, one just needs to check which of the graphs given in \cite[Theorem 6.8 (3.)]{HMR-new} are vertex-transitive. It is easily seen that none of the graphs $F_n$ is vertex-transitive, since there are two types of vertices, those that belong to 4-cycles, and those that don't belong to $4$-cycles.
Hence the following theorem is a direct consequence of \cite[Theorem 6.8]{HMR-new}.

\begin{theorem}\label{thm:VT-3valent}
Let $\Gamma$ be a connected cubic vertex-transitive graph. Then, the following statements are equaivalent:
\begin{enumerate}
\item There exists a strong clique in $\Gamma$;
\item $\Gamma$ is localizable;
\item $\Gamma$ is isomorphic to one of the graphs $K_4$, $K_{3,3}$ or $\overline{C_6}$.
\end{enumerate}
\end{theorem}

\subsection{Valency $4$}

In this section we classify all connected 4-valent vertex-transitive graphs admitting a strong clique. Before stating the result, we need one definition.

\begin{definition}
Let $n\geq 2$ be a positive integer. With $H_n$ we denote the graph of order $4n$ with vertex set $V(H_n)=\{x_i,y_i,z_i,w_i\mid i \in \{1,\ldots,n\}\}$ and   adjacencies are defined as follows: for every $i$, vertices $x_i,y_i,z_i$ and $w_i$ are adjacent, and for $i=1,\ldots n-2$ we have $z_i\sim x_{i+1}$, $w_i\sim y_{i+1}$ and $x_1\sim z_n$ and $y_1\sim w_n$ (see Figure~\ref{fig:H_4} for an example).
\end{definition}

With $\Gamma_1\square \Gamma_2$ we denote the {\em Cartesian product} of graphs $\Gamma_1$ and $\Gamma_2$, that is the graph with vertex set $V(\Gamma_1)\times V(\Gamma_2)$, where two vertices $(x_1,y_1)$ and $(x_2,y_2)$ are adjacent if and only if $x_1=x_2$ and $\{y_1,y_2\}\in E(\Gamma_2)$, or $y_1=y_2$ and $\{x_1,x_2\}\in E(\Gamma_1)$.

\begin{theorem}\label{thm:4valent-main}
A connected 4-valent vertex-transitive graph admits a strong clique if and only if it is isomorphic to one of the graphs $K_{4,4}$, $K_5$, $K_3[2K_1]$, $L(K_{3,3})$, $H_n$, $Cay(\mathbb{Z}_{3k},\{\pm 1, \pm k\})$ or $C_3\square C_n$. In particular, every 4-valent vertex-transitive graph admitting a strong clique is localizable.
\end{theorem}
\begin{proof}
It is straightforward to verify that each of the graphs $K_{4,4}$, $K_5$, $K_3[2K_1]$, $L(K_{3,3})$, $H_n$, $Cay(\mathbb{Z}_{3k},\{\pm 1, \pm k\})$ or $C_3\square C_n$ admits a strong clique and is localizable, so we omit the details.

Let $\Gamma$ be a connected 4-valent vertex-transitive graph and let $C_1$ be a strong clique in $\Gamma$. If $|C_1|=2$, then by Proposition~\ref{prop:trianglefree-regular} it follows that $\Gamma\cong K_{4,4}$. If $|C_1|=5$, then since $\Gamma$ is connected and $4$-valent, it follows that $\Gamma\cong K_5$. 

Suppose now that $|C_1|=4$ and let $C_1=\{x_1,y_1,z_1,w_1\}$. Since $\Gamma$ is 4-valent, each vertex of $C_1$ has another neighbour lying outside of $C_1$.
We claim that no two vertices of $C_1$ share a common neighbour outside of $C_1$. Suppose contrary, that $u\not \in C_1$ is a common neighbour of $x_1$ and $y_1$. Then $x_1$ and $y_1$ have the same closed neighbourhoods in $\Gamma$. It follows that $\Gamma\cong \Gamma'[K_t]$, where $t\geq 2$. Considering the valency of $\Gamma'[K_t]$ it follows that $t\cdot  val(\Gamma')+t-1=4$. Therefore, $t\cdot  (val(\Gamma')+1)=5$, and since $t\geq 2$, the only solution is $t=5$ and $val(\Gamma')=0$, which is impossible, since $\Gamma$ would be disconnected in this case. This proves our claim that no two vertices of $C_1$ share a common neighbour outside of $C_1$. 

Let $x$, $y$, $z$ and $w$ be the remaining neighbours of $x_1,y_1,z_1$ and $w_1$, respectively. If $\{x,y,z,w\}$ is an independent set in $\Gamma$, then Lemma~\ref{lem:dominates} implies that $C_1$ is not a strong clique. We may without loss of generality assume that $x$ is adjacent with $y$. Observe that $\{x,x_1\}$ is a maximal clique in $\Gamma$, and since it is not maximum, by Theorem~\ref{thm:strongclique} it is not strong clique in $\Gamma$. Therefore, $x_1$ lies in a unique strong clique, and exactly three edges incident with $x_1$ lie in a strong clique. Since $\Gamma$ is vertex-transitive, the same must hold for every vertex. 
We claim that $w$ and $z$ are adjacent. Consider the subgroup of $Aut(\Gamma)$ that fixes $C_1$ setwise in its induced action on $C_1$. It is transitive group of degree $4$, hence it admits a regular subgroup $H$ isomorphic to $\ZZ_2\times \ZZ_2$ or $\ZZ_4$. Therefore, there exists an automorphism $\varphi$ of $\Gamma$ which fixes $C_1$ setwise, maps $x_1$ to $w_1$ and acts semiregularly on $C_1$. Then it is easy to see that $\varphi$ induces one of the following three permutations of $C_1$: $(x_1 \, w_1)(y_1\, z_1)$, $ (x_1 \, w_1\,y_1\, z_1)$ or $ (x_1 \, w_1\,z_1\, y_1)$. 
Since $\{x,x_1\}$,  $\{y,y_1\}$, $\{z,z_1\}$ and $\{w,w_1\}$ are the unique maximal cliques of size 2 containing $x_1,y_1,z_1$ and $w_1$, respectively, then it follows that $\varphi$ fixes also the set $\{x,y,z,w\}$, and induces one of the following three permutations $(x \, w)(y\, z)$, $ (x \, w\,y\, z)$ or $ (x \, w\,z\, y)$. In each of the three cases we obtain that edge $\{x,y\}$ is mapped to $\{z,w\}$, either by $\varphi$ or $\varphi^2$, which proves our claim that $z$ and $w$ are adjacent.

Let $C_2$ be the unique strong clique containing $z$. Since the edge $\{z,z_1\}$ doesn't lie in a strong clique, it follows that $w\in C_2$. 
Since edge $\{x,y\}$ lies in a strong clique, it follows that vertices $x$ and $y$ belong to the same strong clique.
If $C_2=\{x,y,z,w\}$, then since $\Gamma$ is connected, it follows that $\Gamma$ is isomorphic to graph $H_2$, and if 
 $C_2\neq \{x,y,z,w\}$ then $C_2\cap \{x,y,z,w\}=\{z,w\}$.

 \begin{figure}
\begin{center}
\includegraphics[scale=0.9]{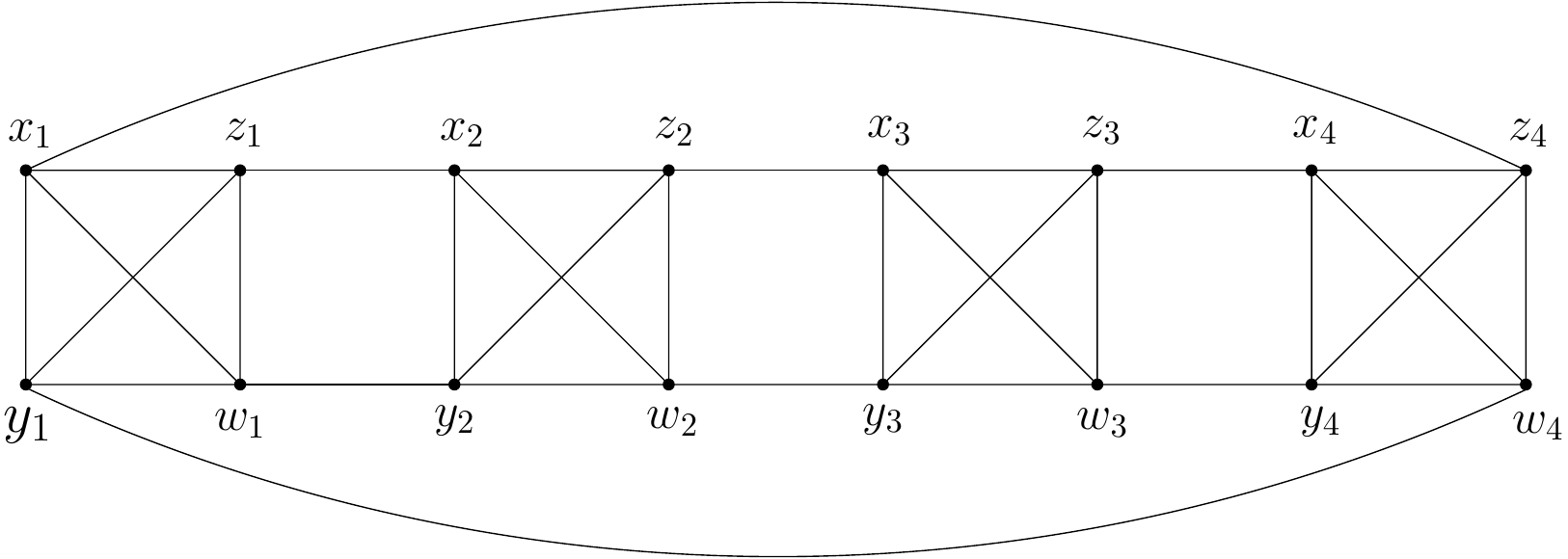} 
\caption{Graph $H_4$}
\label{fig:H_4}
\end{center}
\end{figure}

To simplify the notation, we denote $z$ with $x_2$, $w$ with $y_2$, and the remaining two vertices of $C_2$ with $z_2$ and $w_2$. 
Let $x_3$ and $y_3$ be the remaining neighbours of $z_2$ and $w_2$, respectively. 
Since each vertex lies in a unique strong clique, and by the arguments from the previous paragraphs, it follows that $x_3$ and $y_3$ are adjacent and belong to the same strong clique. Let $C_3$ be the strong clique containing $x_3$ and $y_3$.
If $C_3$ contains one of $x$ and $y$, then it must contain both of them, that is $C_3=\{x_3,y_3,x,y\}$ which implies that $\Gamma\cong H_3$ and we are done.  Therefore, we may assume that $C_3\cap \{x,y\}=\emptyset$. Let $z_3$ and $w_3$ be the remaining vertices of $C_3$, and let $x_4$ and $y_4$ be their remaining neighbours, respectively. Then using the same arguments as before, $x_4$ and $y_4$ are adjacent and belong to strong clique $C_4$. If $C_4=\{x_4,y_4,x,y\}$ then $\Gamma\cong H_4$, otherwise we repeat this procedure. Since $\Gamma$ is finite, this process has to terminate after finitely many steps. Hence $\Gamma\cong H_n$, for some natural number $n$.

{\bf Case $|C_1|=3$}~

Finally suppose that $|C_1|=3$. Suppose first that $\Gamma$ is not irreducible. Then $\Gamma \cong \Gamma'[tK_1]$. Since $\Gamma$ is 4-valent, it is easy to see that $t=2$ and $\Gamma\cong K_3[2K_1]$. We will now assume that $\Gamma$ is irreducible. By Lemma~\ref{lem:irreducible} it follows that two strong cliques are either disjoint or intersect in a single vertex. This implies that each vertex belongs to 1 or 2 strong cliques. Suppose first that some vertex $u$ belongs to 2 strong cliques. Then every edge incident with $u$ belongs to a strong clique. Hence, every maximal clique containing $u$ is strong, and since $\Gamma$ is vertex-transitive, it follows that all maximal cliques in $\Gamma$ are strong. Therefore, $\Gamma$ is CIS graph and by \cite[Corollary 5.6]{MR3278773} it follows that $\Gamma\cong L(K_{3,3})$. 
 
 Now we suppose that each vertex of $\Gamma$ belongs to a unique strong clique. We colour an edge of $\Gamma$ {\it blue} if it belongs to a strong clique, and {\it red} otherwise. It is clear that each vertex is incident with $2$ red and $2$ blue edges. It is also clear that automorphisms of $\Gamma$ preserve such defined colours of edges.

Let $v_0,\ldots ,v_{n-1},v_0$ be a cycle  induced by the red edges  of $\Gamma$.
Suppose that there exists a blue edge joining two vertices of this red cycle. Without loss of generality, assume that $v_0$ is adjacent to $v_{k}$ via blue edge. 
Suppose first that there exists automorphism of $\Gamma$ acting as a $1$-step rotation of the cycle $v_0,\ldots,v_{n-1},v_0$. Without loss of generality we may assume that this rotation maps $v_i$ to $v_{i+1}$. It follows that $v_0,v_{k},v_{2k}$ has to induce a triangle. This implies that $n=3k$, and since $\Gamma$ is connected and $4$-valent, there are no other vertices of $\Gamma$. Therefore, we conclude that $\Gamma\cong Cay(\mathbb{Z}_{3k},\{\pm 1, \pm k\})$.

Suppose now that there is no permutation acting as a $1$-step rotation of this red cycle. Then $n$ must be even, and there exists dihedral subgroup acting regularly on $\{v_0,\ldots ,v_{n-1}\}$, consisting of even-step rotations and reflections without fixed points. If $k$ is even, then since there exists automorphism of $\Gamma$ that acts on $\{v_0,\ldots ,v_{n-1}\}$ as $k$-step rotation, it follows that $v_0,v_k,v_{2k}$ is triangle, and hence $n=3k$. Then for every $i$, it follows that $v_{2i},v_{2i+k},v_{2i+2k}$ is triangle. Now using reflections, it follows that for every $j$ we have triangles of the form $v_{j},v_{j+k},v_{j+2k}$. Since $\Gamma$ is $4$-valent and connected, it follows that $\Gamma\cong Cay(\mathbb{Z}_{3k},\{\pm 1, \pm k\})$. 

Suppose next that for every blue edge $\{v_i,v_j\}$ we have that $i-j$ is odd. This implies that each of the vertices $v_0,\ldots,v_{n-1}$ is adjacent to exactly one blue edge with another vertex from the same red cycle, since otherwise we would obtain triangle $\{v_i,v_j,v_t\}$ and it is impossible that each of the numbers  $i-j$, $j-t$ and $t-i$ is odd. Hence there are exactly $n/2$ blue edges in the subgraph of $\Gamma$ induced by $v_0,\ldots,v_{n-1}$ and they induce a perfect matching. The same must hold for any other cycle induced by the red edges. However, this is impossible, since distinct red cycles are vertex disjoint and blue edges must form triangles.

Finally assume that there is no blue edge joining two vertices on the same red cycle. Suppose that $v_1$, $u_1$ and $z_1$ form a blue triangle. Let $u_1,\ldots,u_n,u_1$ and $z_1,\ldots,z_n,z_1$ be the red cycles containing $u_1$ and $z_1$ respectively. Since $\{v_1,u_1,z_1\}$ is strong clique in $\Gamma$ it follows that none of the sets $\{v_j,u_k,z_l\}$ is independent, where $j,k,l\in \{2,n\}$. Let $H$ be the subgraph of $\Gamma$ induced by the vertices $v_2,u_2,z_2,v_n,u_n,z_n$. We claim that no vertex is isolated in $H$. Suppose contrary, that $v_2$ is isolated in $H$. Then based on the above observation, it follows that $u_2$ is adjacent with both $z_2$ and $z_n$. However, since the blue edges induce triangles, it would imply that $z_2$ and $z_n$ are adjacent, contrary to the above assumption that no blue edge joins two vertices on the same red cycle. This proves our claim that no vertex of $H$ is isolated. Next we claim that no vertex of $H$ has degree $1$. Suppose contrary, that $v_2$ has degree $1$ in $H$, and without loss of generality, that $u_2$ is its neighbour. Then $u_n$ has to be adjacent with both $z_2$ and $z_n$, and consequently $z_2$ and $z_n$ are adjacent, contrary to the above assumption. This proves our claim that no vertex of $H$ is of degree $1$. Since each vertex of $\Gamma$ has exactly two blue edges incident with it, and since no red edge of $\Gamma$ lies in $H$, it follows that each vertex of $H$ has degree $2$. Since each blue edge lies in a triangle, it follows that $H$ is a disjoint union of $2$ triangles. Without loss of generality, we may assume that $\{v_2,u_2,z_2\}$ and $\{v_n,u_n,z_n\}$ are the two triangles. Now applying the same arguments, it follows that $\{v_i,u_i,z_i\}$ induces a triangle, for each $i$. By the connectedness of $\Gamma$ it follows that $V(\Gamma)=\{v_i,u_i,z_i\mid i\in \{1,n\}\}$. It is now easy to see that $\Gamma\cong Cay(\ZZ_n\times \ZZ_3, \{(\pm 1,0), (0,\pm 1)\})$. Observe that this graph is isomorphic to $C_3\square C_n$.
 \end{proof}

\subsection{Valency $5$}

The goal of this section is to give a characterization of $5$-valent vertex-transitive graphs admitting a strong clique. We first show that there is a unique example of a connected 5-valent vertex-transitive graph with given clique number $\omega\neq 4$.

\begin{proposition}\label{prop:5 valent omega=2,3,5,6}
Let $\Gamma$ be a connected $5$-valent vertex-transitive graph with $\omega(\Gamma)\neq 4$. Then $\Gamma$ admits a strong clique if and only if 
\begin{enumerate}[(i)]
\item $\omega(\Gamma)=2$ and $\Gamma \cong K_{5,5}$;
\item $\omega(\Gamma)=3$ and $\Gamma \cong Cay(\ZZ_{12},\{\pm 1, \pm 4, 6\})$;
\item $\omega(\Gamma)=5$ and $\Gamma \cong K_5\square K_2$;
\item $\omega(\Gamma)=6$ and $\Gamma \cong K_6$.
\end{enumerate}
\end{proposition}

\begin{proof}
If $\omega(\Gamma)=2$, the result follows by Proposition~\ref{prop:trianglefree-regular}. If $\omega(\Gamma)=6$, then the result follows trivially. 
Suppose now that $\omega(\Gamma)=5$. 
Let $a$ be a vertex of $\Gamma$ and let $C=\{a,b,c,d,e\}$ be a strong clique containing $a$. Let $a_1$ be the remaining neighbour of $a$. If $a_1$ is not isolated vertex in the subgraph induced by $\Gamma(a)$, then there $a$ would have the same closed neighbourhood as one of $b,c,d,f$, which implies that $\Gamma=\Gamma'[K_t]$, for some $t\geq w$. However, since $\Gamma$ is $5$-valent and $\omega(\Gamma)=5$ it is easy to see that this is impossible. 
It follows that $\Gamma(a)$ is isomorphic to disjoint union of $K_4$ and $K_1$, hence every vertex belongs to unique clique of size $5$. Since every vertex belongs to unique strong clique, it follows that there exists a cyclic group of order $5$, fixing $C$ setwise, and acting transitively on the vertices of $C$. Without loss of generality, we may assume that there exists $\alpha\in Aut(\Gamma)$ whose restriction to $C$ is $5$-cycle $(a~b~c~d~e)$. Let $b_1=\alpha(a_1)$, $c_1=\alpha(b_2)$, $d_1=\alpha(c_1)$ and $e_1=\alpha(d_1)$. Since $\alpha$ is automorphism of $\Gamma$ it follows that $b\sim b_1$, $c\sim c_1$, $d\sim d_1$ and $e\sim e_1$. Since $C$ is strong clique, it follows that $\{a_1,b_1,c_1,d_1,e_1\}$ is not an independent set in $\Gamma$. Using the fact that the local graph at every vertex is isomorphic to disjoint union of $K_1$ and $K_4$, it follows that $\{a_1,b_1,c_1,d_1,e_1\}$ is a clique in $\Gamma$. Since $\Gamma$ is connected $5$-valent graph, it follows that the order of $\Gamma$ is 10, and it is now easy to see that $\Gamma\cong K_5\square K_2$.

Suppose now that $\omega(\Gamma)=3$ and let $C$ be a strong clique in $\Gamma$. Since $\omega(\Gamma)=3$, by Corollary~\ref{cor:strongclique} it follows that $C$ is maximum clique in $\Gamma$, hence $|C|=3$. Let $C=\{x,y,z\}$. If $\Gamma$ is not irreducible, then since it is vertex-transitive of valency $5$, $\Gamma$ would have to be isomorphic to $K_{5,5}$ which is impossible, since $\omega(\Gamma)=3$. Therefore, $\Gamma$ is irreducible. By Lemma~\ref{lem:irreducible} it follows that each vertex belongs to one or two strong cliques. 

Suppose first that each vertex belongs to two strong cliques.
Let $\alpha(\Gamma)=n$. By Theorem~\ref{thm:strongclique} it follows that $|V(\Gamma)|=3n$.  Let $T$ denote the set of all triangles of $\Gamma$.
Since each vertex belongs to two strong cliques, it follows $|T|=2n$. 
Let $G\leq \Aut(\Gamma)$ be a vertex-transitive group of automorphisms and consider the induced action of $G$ on $T$. Since each vertex of $\Gamma$ belongs to two triangles, and $G$ acts transitively on $V(\Gamma)$ it follows that either $G$ acts transitively on $T$, or it has two orbits, each of size $n$. If $G$ acts transitively on $T$, then $2n=|T|\mid |G|$. Since $G$ also acts transitively on $V(\Gamma)$, it follows that $6n\mid |G|$, and consequently $3\mid |G_C|$, by the orbit-stablizer property. Similarly, if $G$ has two orbits in its action on $T$, then again $3\mid |G_C|$. In any case, there exists an element $g\in G$ of order $3$ such that $g$ fixes $C$ setwise. It follows that $g$ either cyclically permutes vertices of $C$, or fixes all three of them. Suppose that $g$ fixes all three vertices of $C$. Since $g$ is of order $3$, it easily follows that $g$ fixes also all the neighbours of the vertices from $C$. The connectedness of $\Gamma$ implies that $g=1$, contrary to the fact that $g$ has order $3$. Therefore, $g$ cyclically permutes the vertices of $C$.
 
 Let $\{x,x_1,x_2\}$, $\{y,y_1,y_2\}$ and $\{z,z_1,z_2\}$ be the remaining strong cliques containing $x$, $y$ and $z$ respectively, and let $x_3$, $y_3$, and $z_3$, respectively, be the unique neighbours of $x$, $y$ and $z$ not contained in a triangle. Without loss of generality, we may assume that $g(x)=y$, and $g(y)=z$. This implies that $g(x_3)=y_3$ and $g(y_3)=z_3$. Without loss of generality, we may assume that $g(x_1)=y_1$ and $g(y_1)=z_1$. It is now clear that also $g(x_2)=y_2$ and $g(y_2)=z_2$.
 
Since $C$ is strong clique,  it follows that $\{x_i,y_i,z_i\}$ is not an independent set for each $i\in \{1,2,3\}$. Existence of one edge with between two vertices in $\{x_i,y_i,z_i\}$, using the fact that $g$ is an automorphism cyclically permuting those vertices, implies that $\{x_i,y_i,z_i\}$ induces a triangle, see Fig.~\ref{fig:5-val}.

\begin{figure}[h]

\begin{center}
\includegraphics[scale=0.9]{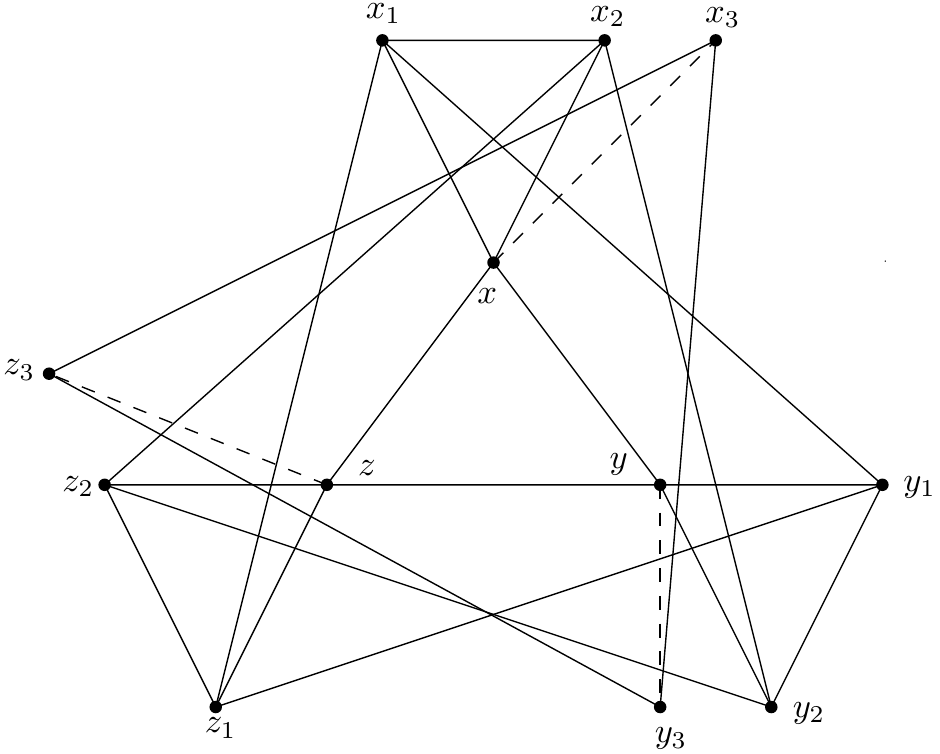} 
\caption{Local structure of $\Gamma$ around $C$. Dashed edges are the edges not contained in a triangle.}
\label{fig:5-val}
\end{center}
\end{figure}

We claim that $\{x_3,y_1,z_2\}$ is an independent set.
Suppose that $x_3$ is adjacent with $y_1$. Since edge $\{x,x_3\}$ does not belong to a triangle, it follows that $\{x_3,y_1\}$ belongs to a triangle. However, $y_1$ already belongs to triangles $\{y,y_1,y_2\}$ and $\{x_1,y_1,z_1\}$. Since each vertex belongs to exactly $2$ triangles, it follows that $\{x_3,y_1\}$ does not belong to a triangle, a contradiction. The obtained contradiction shows that $x_3$ is not adjacent with $y_1$, and by analogy, $x_3$ is not adjacent with $z_2$, as well. If $y_1$ and $z_2$
are adjacent, then again, we obtain third triangle containing $y_1$, namely $\{y_1,z_1,z_2\}$. Therefore,   $\{x_3,y_1,z_2\}$ is an independent set, and it dominates $C$, contrary to the assumption that $C$ is a strong clique. The obtained contradiction shows that there is no 5-valent vertex-transitive graph admitting a strong clique, with a vertex belonging to two strong cliques.

Suppose now that each vertex of $\Gamma$ belongs to a unique strong clique. Denote the remaining neighbours of $x$, $y$, and $z$, with $x_i$, $y_i$, and $z_i$, respectively, where $i\in \{1,2,3\}$.
Since each vertex belongs to a unique strong clique, it follows that each strong clique is a block of imprimitivity for the action of $G$. Therefore, there exists an automorphism $g\in \Aut(\Gamma)$ cyclically permuting the vertices of $C$. Without loss of generality, we may assume that $g(x)=y$ and $g(y)=z$. Arguing similarly as in the previous case, we may assume that for each $i\in \{1,2,3\}$, $g(x_i)=y_i$, $g(y_i)=z_i$, and  $\{x_i,y_i,z_i\}$ forms a triangle in $\Gamma$. 
 Since $C$ is strong clique, it follows that $\{x_1,y_2,z_3\}$ is not an independent set, and without loss of generality, we may assume that $x_1$ is adjacent with $y_2$. By the action of $g$, it follows that also $y_1$ is adjacent to $z_2$, and $z_1$ is adjacent to $x_2$.  Observe that $\{x_2,y_1\} \not \in E(\Gamma)$, since otherwise we obtain that $x_1$ belongs to another triangle $\{x_2,y_1,z_1\}$. 
Since $C$ is a strong clique, it follows that $\{x_2,y_1,z_3\}$ is not an independent set in $\Gamma$. Therefore, we have that $\{x_2,z_3\}\in E(\Gamma)$ or $\{y_1,z_3\}\in E(\Gamma)$. 
Without loss of generality, we assume that $\{y_1,z_3\}\in E(\Gamma)$. The action of $g$ implies that also $\{z_1,x_3\}, \{x_1,y_3\}\in E(\Gamma)$. 
Observe that $\{x_1,z\}$ is an independent set in $\Gamma$ which dominates $y_2$ and $z_2$. Since $\{x_2,y_2,z_2\}$ is a strong clique it follows that $\Gamma(x_2)\subset \Gamma(x_1)\cup \Gamma(z)$. Observe that $\Gamma(x_1)=\{x,y_1,z_1,y_2,y_3\}$ and $\Gamma(z)=\{x,y,z_1,z_2,z_3\}$. It follows that $x_2$ must be adjacent with one of $y_3$ or $z_3$. 
If $\{x_2,y_3\}\in E(\Gamma)$ it follows that $\{y_2,z_3\},\{z_2,x_3\}\in E(\Gamma)$, and since $\Gamma$ is connected, it follows that there are no other vertices in $\Gamma$. It is not difficult to verify that in this case $\Gamma\cong Cay(\ZZ_{12},\{\pm 1, \pm 4,6\})$.
The same result follows if $\{x_2,z_3\}\in E(\Gamma)$. This concludes the proof.
\end{proof}

\begin{corollary}
Let $\Gamma$ be a 5-valent vertex-transitive graph with $\omega(\Gamma)\neq 4$. Then $\Gamma$ admits a strong clique if and only if it is localizable. 
\end{corollary}

\subsubsection{Valency 5 with clique number 4}\label{sec:valency 5 omega=4}
We now turn our attention to the study of strong cliques in $5$-valent vertex-transitive graph with clique number equal to $4$. 

\begin{lemma}\label{lem:5 valent omega=4}
Let $\Gamma$ be a connected $5$-valent vertex-transitive graph with $\omega(\Gamma)=4$. Then $\Gamma$ is irreducible. Furthermore, if local graph of $\Gamma$ has a universal vertex, then $\Gamma$  admits a strong clique if and only if it is isomorphic to $ C_4[K_2]$.
\end{lemma}
\begin{proof}
Suppose that $\Gamma$ is a reducible connected $5$-valent vertex-transitive graph with $\omega(\Gamma)=4$. 
Then $\Gamma\cong X[tK_1]$, where $X$ is a vertex-transitive graph, and $t\geq 2$. Since $\Gamma$ is $5$-valent it follows that $t=5$, and therefore $\Gamma\cong K_2[5K_1]\cong K_{5,5}$, which is impossible, since $\omega(\Gamma)=4$.

Suppose now that the local graph of $\Gamma$ has a universal vertex. It follows that $\Gamma\cong X[K_t]$, for some vertex-transitive graph $X$, and positive integer $t\geq 2$. Since $\Gamma$ is $5$-valent and $\omega(\Gamma)=4$, it is easy to see that $t=2$. Hence $\Gamma \cong C_n[K_2]$. It is not difficult to see that $C_n[K_2]$ admits a strong clique if and only if $n=4$.  
\end{proof}

Let $\Gamma$ be an irreducible $5$-valent vertex-transitive graph with $\omega(\Gamma)=4$ admitting a strong clique, and let $L$ denote the local graph of $\Gamma$. Suppose that $L$ doesn't have a universal vertex. Since $\omega(\Gamma)$ is $4$, it follows that $\omega(L)=3$, and  by Lemma~\ref{lem:irreducible} it follows that no two cliques of size $3$ in $L$ share an edge. Since $L$ has no universal vertex, it follows that $L$ has a unique clique of size $3$. Proof of the following simple lemma is straightforward and is thus omitted.

\begin{lemma}\label{lem:local graphs}
Let $L$ be a graph of order $5$ having exactly one clique of size $3$ and no universal vertex. Then $L$ is isomorphic to one of the graphs shown in Figure~\ref{fig:five}.
\end{lemma}

 \begin{figure}[h!]
\begin{center}
\includegraphics[scale=0.9]{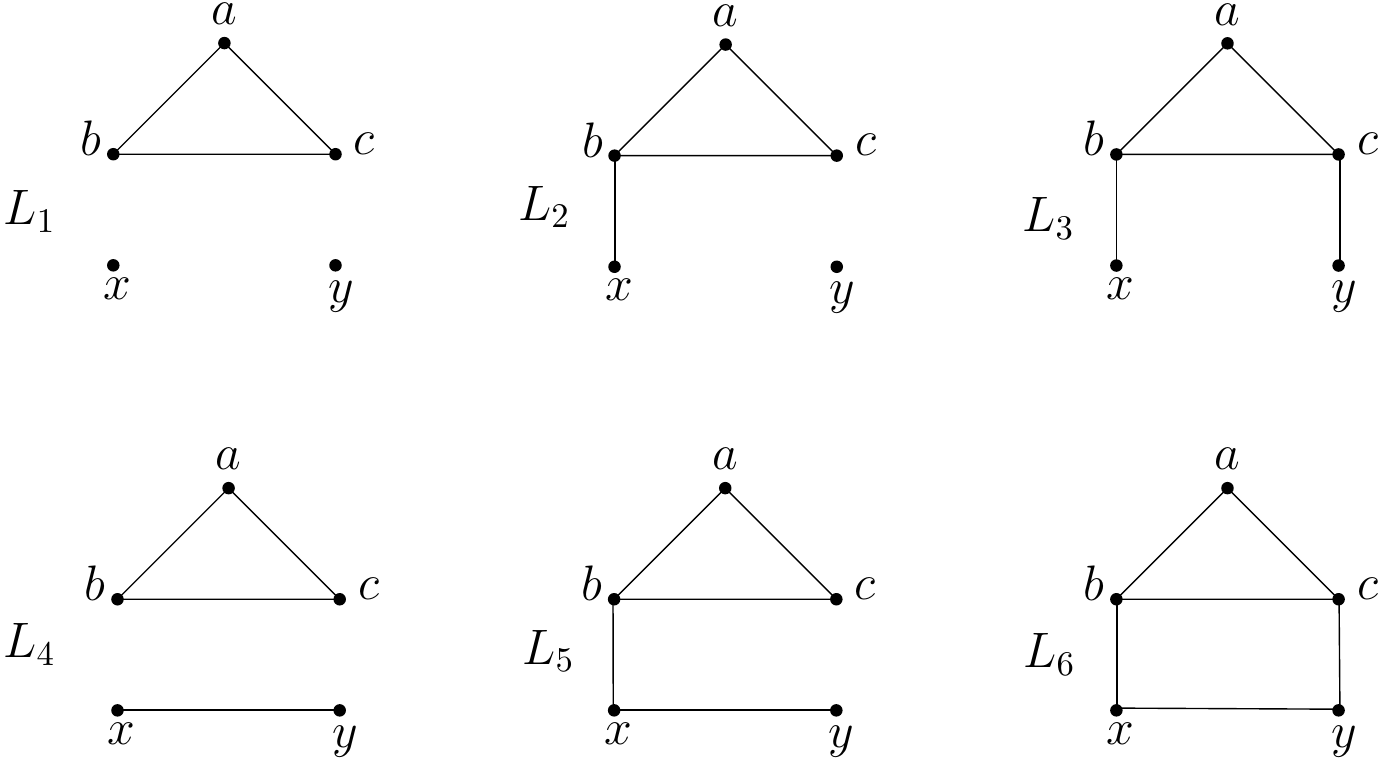} 
\caption{Local graphs of $5$-valent vertex-transitive with $\omega=4$ admitting a strong clique}
\label{fig:five}
\end{center}
\end{figure}

Proposition~\ref{prop:5 valent omega=2,3,5,6} and Lemmas \ref{lem:irreducible}, \ref{lem:5 valent omega=4} and \ref{lem:local graphs} imply that the study of existence of strong cliques in $5$-valent vertex-transitive graphs is reduced to the study of vertex-transitive graphs with local graph being isomorphic to one of the graphs $L_i$ shown on Figure~\ref{fig:five}. 
In the following lemma we give an infinite family of graphs (including graphs $L_2$ and $L_3$) that can not be realised as local graphs of a vertex-transitive graph.

\begin{lemma}
Let $X$ be a graph with unique maximum clique $C$. If there exists an edge in $X$ not contained in $C$, and every edge in $X$ has non-empty intersection with $C$, then there is no vertex-transitive graph with local graph isomorphic to $X$.
\end{lemma}
\begin{proof}
Suppose that $\Gamma$ is vertex-transitive graph and the local graph of $\Gamma$ at $v$ is isomorphic to $X$.
Let $C$ be the unique maximum clique of $X$, and let $\{x,y\}$ be an edge with $x\in C$ and $y\not \in C$. By the assumption on $X$ it follows that each vertex in $\Gamma$ belongs to unique maximum clique in $\Gamma$. Observe that $\{v,x\}$ is an edge in $\Gamma(y)$. Let $C'$ be the unique maximum clique of $\Gamma(y)$. By the hypothesis, it follows that $v\in C'$ or $x\in C'$. It follows that $C'\cup \{x\}$ or $C'\cup \{v\}$ is a maximum clique in $\Gamma$, contrary to the assumption that $C\cup \{v\}$ is the unique maximum clique in $\Gamma$ containing $v$ and $x$.
\end{proof}

\begin{corollary}
There is no vertex-transitive graph with local graph isomorphic to $L_2$ or $L_3$.
\end{corollary}

\begin{lemma}
There is no vertex-transitive graph with local graph isomorphic to $L_5$.
\end{lemma}
\begin{proof}
Let $\Gamma$ be a vertex-transitive graph with $\Gamma(v)\cong L_5$. Let $z$ be the remaining neighbour of $b$. Since $\Gamma(b)\cong L_5$ it follows that $z$ is adjacent with $x$ and non-adjacent with $c$ and $a$.
Let $w$ be the remaining neighbour of $x$. It is not difficult to see that $\Gamma(x)\not \cong L_5$, a contradiction.
\end{proof}

\begin{lemma}
The only connected vertex-transitive graph with local graph isomorphic to $L_6$ is $\overline{C_8}$.
\end{lemma}
\begin{proof}
Let $\Gamma$ be a vertex-transitive graph with $\Gamma(v)\cong L_6$.
Let $z$ and $w$ be the remaining neighbours of $a$. Since $\Gamma(a)\cong L_6$ it follows that $w\sim z$ and without loss of generality we may assume that $b\sim w$ and $c\sim z$. Considering $\Gamma(b)$ it follows that $x\sim w$ and considering $\Gamma(c)$ it follows that $y\sim z$. Since $\Gamma$ is connected and $5$-valent, it follows that there are no other vertices of $\Gamma$. It is easy to see that $\overline{\Gamma}$ is $8$-cycle $vwcxaybzv$, hence $\Gamma\cong \overline{C_8}$.
\end{proof}

In the following lemma we show that there is a unique connected vertex-transitive graph with local graph isomorphic to $L_4$ admitting a strong clique.
\begin{lemma}
Let $\Gamma$ be a connected vertex-transitive graph with local graph isomorphic to $L_4$. Then $\Gamma$ admits a strong clique if and only if $\Gamma\cong K_3\square K_4$. 
\end{lemma}
\begin{proof}
Let $\Gamma$ be a connected $5$-valent vertex-transitive graph with local graph isomorphic to $L_4$ and suppose that $\Gamma$ admits a strong clique. Let us call an edge of $\Gamma$ {\it blue} if it belongs to clique of size $4$ and {\it red} otherwise. It is clear that the automorphisms of $\Gamma$ preserve such defined colors of edges.  Since the local graph of $\Gamma$ is isomorphic to $L_4$ it follows that every vertex is adjacent with two red edges. 
Let $C=\{x,y,z,w\}$ be a strong clique in $\Gamma$ and let $x_i,y_i,z_i,w_i$ ($i=1,2$) be the remaining neighbours of $x,y,z$ and $w$, respectively. Considering the local graphs at $x,y,z$ and $w$, it follows that $x_1\sim x_2$, $y_1\sim y_2$, $z_1\sim z_2$, and $w_1\sim w_2$, and all of these edges are red.
Let $K$ be the subgroup of $Aut(\Gamma)$ that fixes $C$ setwise in the induced action on $C$. Since $C$ is block for $Aut(\Gamma)$, it follows that $K$ acts transitively on $C$.
There exists a subgroup $H$ of $K$ transitive on $C$ isomorphic either to $\ZZ_4$ or $\ZZ_2\times \ZZ_2$. We will first suppose that $H\cong \ZZ_4$.

Let $\Gamma_C=\Gamma[\{x_1,x_2,y_1,y_2,z_1,z_2,w_1,w_2\}]$.
It is clear that the only red edges in the graph $\Gamma_C$ are $x_1x_2$, $y_1y_2$, $z_1z_2$ and $w_1w_2$.
Suppose that $x_1$ has degree $4$ in $\Gamma_C$.
Without loss of generality, we may assume that $x_1$ is adjacent with $x_2,y_1,z_1,w_1$. Then considering the local graph at $x_1$, it follows that $\{x_1,y_1,z_1,w_1\}$ is a clique in $\Gamma$.
Since $C$ is a strong clique, it follows that $\{x_2,y_2,z_2,w_2\}$ is not an independent set. If some of its vertices has degree 4 in $\Gamma_C$, then it follows that also $\{x_2,y_2,z_2,w_2\}$ is a clique in $\Gamma$. Suppose that none of the vertices $x_2,y_2,z_2,w_2$ is of degree 4 in $\Gamma_C$. Since  $\{x_2,y_2,z_2,w_2\}$  is not independent set, we may assume that $x_2\sim y_2$. Recall that $H$ cyclically permutes vertices of $C$. Then it also cyclically permutes vertices $x_2,y_2,z_2$ and $w_2$. Using this it follows that $\{x_2,y_2,z_2,w_2\}$ is indeed a clique in $\Gamma$, which implies that $\Gamma \cong K_3\square K_4$.

Suppose now that the maximum degree in $\Gamma_C$ is 3, and let $x_1$ be adjacent with $x_2$, $y_1$ and $z_1$. This implies that $y_1\sim z_1$ (by considering the local graph at $x_1$). 
Using the cyclic group $H$, it follows that there 
are 4 triangles formed by blue edges in the graph $\Gamma_C$. However, since there are $8$ vertices in $\Gamma_C$, these triangles cannot be vertex disjoint, hence at least one vertex would have valency at least $4$ in $\Gamma_C$, a contradiction.

It is easy to see that the maximum degree in $\Gamma_C$ cannot be 2, since there would exist an independent set of size $4$ disjoint with $C$ and  dominating the clique $C$, contrary to the assumption that $C$ is strong. This shows that the only possibility when $H\cong C_4$ is that $\{x_1,y_1,z_1,w_1\}$ and $\{x_2,y_2,z_2,w_2\}$ are cliques in $\Gamma$. 

Suppose now that $H\cong \ZZ_2\times \ZZ_2$. If maximal degree in $\Gamma_C$ is $2$, then there exists an independent set of size $4$ disjoint from $C$ which dominates $C$. Using the argument similar to the case when $H$ is cyclic, it can be seen that the maximum degree in $\Gamma_C$ cannot equal to $3$. Hence the maximal degree in $\Gamma_C$ equals to $4$. Let $x_1$ be adjacent with $x_2,y_1,z_1$ and $w_1$. Then by considering the local graph at $x_1$ it follows that $\{x_1,y_1,z_1,w_1\}$ is clique. If some of the vertices $x_2,y_2,z_2$ or $w_2$ has degree $4$, it follows that also $\{x_2,y_2,z_2,w_2\}$ is clique in $\Gamma$.
If $x_2$ has degree $3$ in $\Gamma_C$, then using the action of $H$ it follows that each of the vertices $y_2$, $z_2$ and $w_2$ has degree $3$ in $\Gamma_C$ which is impossible.

\begin{figure}[h]

\begin{center}
\includegraphics[scale=0.9]{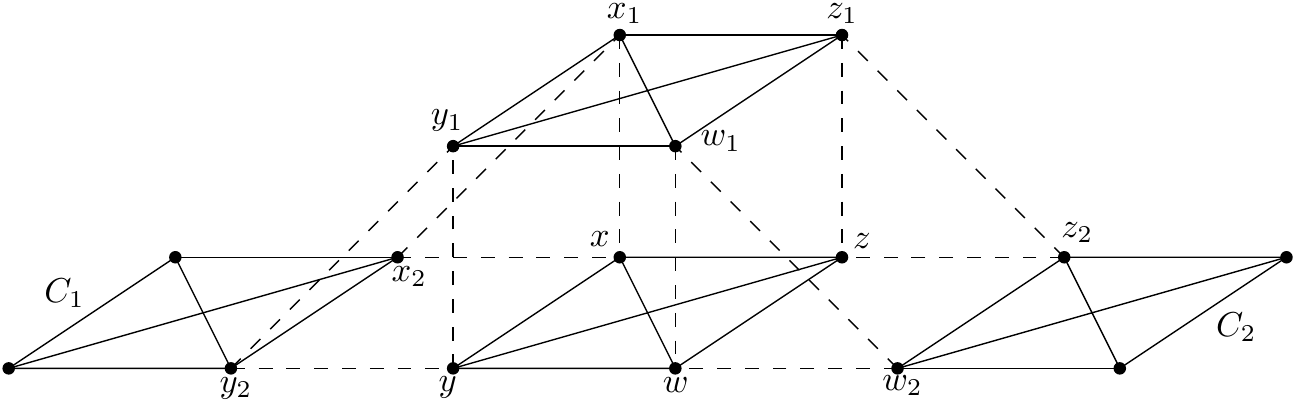} 
\caption{Local structure of $\Gamma$ around $C$. Dashed edges are the edges not contained in a strong clique.}
\label{fig:localL1}
\end{center}
\end{figure}

Finally, suppose that $x_2$ has degree $2$ in $\Gamma_C$. Without loss of generality, we may assume that $x_2\sim y_2$ and $z_2\sim w_2$. Let $C_1$ be the strong clique in $\Gamma$ containing the edge $x_2y_2$ and $C_2$ be the strong clique in $\Gamma$ containing the edge $z_2w_2$. Observe that there are $3$ strong cliques in $\Gamma$ at distance $1$ from $C$ (they are $C_1$, $C_2$ and $\{x_1,y_1,z_1,w_1\}$). 
Since $\Gamma$ is vertex-transitive, and each vertex belongs to unique strong clique, it follows that $\Aut(\Gamma)$ acts transitively on strong cliques, hence all strong cliques have the same properties.
Observe that $C$ is adjacent to $C_1$ and $C_2$ via two disjoint edges $xy$ and $zw$ (see Figure~\ref{fig:localL1}). However, the clique $C_1$ is joined with $C$ and $\{x_1,y_1,z_1,w_1\}$ via single edge $x_2y_2$. This is contradiction, since $Aut(\Gamma)$ acts transitively on strong cliques.
The obtained contradiction shows that for every clique of size $4$ in $\Gamma$, its neighbourhood is disjoint union of two cliques of size $4$. Since $\Gamma$ is connected $5$-valent, it is now easy to see that $\Gamma$ has 12 vertices and it is isomorphic to $K_3\square K_4$.
\end{proof}

The situation when the local graph is $L_1$ is more complicated. Just to demonstrate it, in the following example we give $4$ infinite families of vertex-transitive graphs admitting strong cliques with local graph being isomorphic to $L_1$, hence we leave the characterization of such graphs for future research.

\begin{example}\label{ex:local graph L1}
Each of the graphs $Cay(\ZZ_{4n},\{\pm 1,\pm n, \pm 2n\})$, $Cay(\ZZ_{2n}\times \ZZ_{2},\{(\pm 1,0 ), (n,0),(0,1),(n,1)\})$, $Cay(\ZZ_{n}\times \ZZ_{4},\{(\pm 1,0),(0,1),(0,2),(0,3)\})$, $H_n \square K_2$ with $n\geq 4$ has local graph isomorphic to $L_1$ and admits a strong clique of size $4$.
\end{example}

To summarize our results on $5$-valent vertex-transitive graphs admitting a strong clique, we have the following proposition.

\begin{proposition}\label{prop:5-valent local graph not L1}
Let $\Gamma$ be a connected $5$-valent vertex-transitive graph with local graph not isomorphic to $L_1$. Then $\Gamma$ admits a strong clique if and only if it isomorphic to one of the graphs $K_6$,  $\overline{C_8}$, $C_4[K_2]$, $K_{5,5}$, $K_5\square K_2$, $K_3\square K_4$ or $Cay(\ZZ_{12},\{\pm 1, \pm 4, 6\})$.
\end{proposition}

We conclude with this section with the following proposition.
\begin{proposition}\label{prop:valency at most 5}
Let $\Gamma$ be a vertex-transitive graph with valency at most $5$. Then $\Gamma$ admits a strong clique if and only if $\Gamma$ is localizable.
\end{proposition}
\begin{proof}
Without loss of generality, we may assume that $\Gamma$ is connected. If $\Gamma$ is of valency 1 or $2$ the result trivially follows.
If $\Gamma$ is $3$-valent, then the result follows by Theorem~\ref{thm:VT-3valent}. When $\Gamma$ is $4$-valent, result follows by Theorem~\ref{thm:4valent-main}. If $\Gamma$ is $5$-valent and $\omega(\Gamma)\neq 4$ then the result follows by Proposition~\ref{prop:5 valent omega=2,3,5,6}.
Proposition~\ref{prop:5-valent local graph not L1} deals with the case when $\Gamma$ is $5$-valent, with clique number $4$ and local graph not isomorphic to $L_1$. 
If $\Gamma$ is vertex-transitive graph with local graph isomorphic to $L_1$, then since $L_1$ has a unique triangle, it follows that the vertex-set of $\Gamma$ can be partitioned into disjoint union of cliques of size $4$ (since each vertex belongs to a unique $4$-clique). Therefore, $\Gamma$ admits a strong clique if and only if it is localizable.
\end{proof}

\section{Some open problems}
\label{sec:problems}

In the study of existence of a strong clique in $5$-valent vertex-transitive graphs, we reduced the problem to the study of vertex-transitive graphs with the prescribed local graph. Then in some cases, we were able to show that there exits no vertex-transitive graph with a given local graph.
This leads to a more general problem of characterizing graphs that can arise as local graphs of vertex-transitive graphs. It is also interesting to investigate what can be said about global properties of the graph by considering local graphs. As an example, it is easy to see that a connected vertex-transitive graph with local graph being asymmetric (that is, with trivial automorphism group) is a GRR (that is, it is a Cayley graph on a group $G$, and its full automorphism group is isomorphic to $G$). Hence we propose the following problem for possible further research.

\begin{problem}
Which graphs can be realised as local graphs of vertex-transitive graphs? In which cases there are only finitely many connected vertex-transitive graphs with a given local graph? 
\end{problem}

In Proposition~\ref{prop:5-valent local graph not L1} we classified all $5$-valent vertex-transitive graphs admitting a strong clique, whose local graph is not isomorphic to $L_1$. We saw in Example~\ref{ex:local graph L1} that there are infinitely many examples of vertex-transitive graphs with local graph isomorphic to $L_1$ admitting a strong clique. Hence we propose the following problem.

\begin{problem}
Classify all $5$-valent vertex-transitive graphs with local graph isomorphic to $L_1$ admitting a strong clique.
\end{problem}

Theorem~\ref{thm:J(7,3,1) is CIS not localizable} shows that there are vertex-transitive graphs with every maximal clique being strong, which are not localizable. However, we saw in Proposition~\ref{prop:valency at most 5} that for vertex-transitive graphs of valency at most $5$, even the existence of a strong clique implies localizablity. Hence it is natural to ask for which other small valencies this might be true.

\begin{problem}
Determine the minimal valency of a vertex-transitive graph admitting a strong clique which is not localizable.
\end{problem}

\section*{Acknowledgments}

This work is supported in part by the Slovenian Research Agency (research program P1-0285 and research projects N1-0032, N1-0038, J1-7051, N1-0062, J1-9110).


\begin{thebibliography}{10}

\bibitem{Alon}
N.~Alon.
\newblock The chromatic number of random {C}ayley graphs.
\newblock {\em European J. Combin.}, 34(8):1232--1243, 2013.

\bibitem{MR715895}
C.~Berge.
\newblock Stochastic graphs and strongly perfect graphs---a survey.
\newblock {\em Southeast Asian Bull. Math.}, 7(1):16--25, 1983.

\bibitem{MR778749}
C.~Berge and P.~Duchet.
\newblock Strongly perfect graphs.
\newblock In {\em Topics on perfect graphs}, volume~88 of {\em North-Holland
  Math. Stud.}, pages 57--61. North-Holland, Amsterdam, 1984.

\bibitem{MR3141630}
E.~Boros, V.~Gurvich, and M.~Milani{\v{c}}.
\newblock On {CIS} circulants.
\newblock {\em Discrete Math.}, 318:78--95, 2014.

\bibitem{Boros2015}
E.~Boros, V.~Gurvich, and M.~Milani\v{c}.
\newblock On equistable, split, {CIS}, and related classes of graphs.
\newblock {\em Discrete Appl. Math.}, 216(part 1):47--66, 2017.

\bibitem{MR2489416}
E.~Boros, V.~Gurvich, and I.~Zverovich.
\newblock On split and almost {CIS}-graphs.
\newblock {\em Australas. J. Combin.}, 43:163--180, 2009.

\bibitem{MR778765}
M.~Burlet and J.~Fonlupt.
\newblock Polynomial algorithm to recognize a {M}eyniel graph.
\newblock In {\em Topics on perfect graphs}, volume~88 of {\em North-Holland
  Math. Stud.}, pages 225--252. North-Holland, Amsterdam, 1984.

\bibitem{Cranston}
D.~W. Cranston and L.~Rabern.
\newblock A note on coloring vertex-transitive graphs.
\newblock {\em Electron. J. Combin.}, 22(2):Paper 2.1, 9, 2015.

\bibitem{MR3278773}
E.~Dobson, A.~Hujdurovi{\'c}, M.~Milani{\v{c}}, and G.~Verret.
\newblock Vertex-transitive {CIS} graphs.
\newblock {\em European J. Combin.}, 44(part A):87--98, 2015.

\bibitem{Godsil}
C.~Godsil and B.~Rooney.
\newblock Hardness of computing clique number and chromatic number for {C}ayley
  graphs.
\newblock {\em European J. Combin.}, 62:147--166, 2017.

\bibitem{MR1829620}
C.~Godsil and G.~Royle.
\newblock {\em Algebraic graph theory}, volume 207 of {\em Graduate Texts in
  Mathematics}.
\newblock Springer-Verlag, New York, 2001.

\bibitem{Green}
B.~Green.
\newblock On the chromatic number of random {C}ayley graphs.
\newblock {\em Combin. Probab. Comput.}, 26(2):248--266, 2017.

\bibitem{MR2755907}
V.~Gurvich.
\newblock On exact blockers and anti-blockers, {$\Delta$}-conjecture, and
  related problems.
\newblock {\em Discrete Appl. Math.}, 159(5):311--321, 2011.

\bibitem{MR1677797}
B.~L. Hartnell.
\newblock Well-covered graphs.
\newblock {\em J. Combin. Math. Combin. Comput.}, 29:107--115, 1999.

\bibitem{MR888682}
C.~T. Ho{\`a}ng.
\newblock On a conjecture of {M}eyniel.
\newblock {\em J. Combin. Theory Ser. B}, 42(3):302--312, 1987.

\bibitem{MR1301855}
C.~T. Ho{\`a}ng.
\newblock Efficient algorithms for minimum weighted colouring of some classes
  of perfect graphs.
\newblock {\em Discrete Appl. Math.}, 55(2):133--143, 1994.

\bibitem{HMR-new}
A.~{Hujdurovi{\'c}}, M.~{Milani{\v c}}, and B.~{Ries}.
\newblock {Detecting strong cliques}.
\newblock \url{https://arxiv.org/abs/1808.08817}, Aug. 2018.

\bibitem{HujdurovicMR2018}
A.~Hujdurovi\'c, M.~Milani\v{c}, and B.~Ries.
\newblock Graphs vertex-partitionable into strong cliques.
\newblock {\em Discrete Math.}, 341(5):1392--1405, 2018.

\bibitem{Ilic}
A.~Ili\'c and M.~Ba\v{s}i\'c.
\newblock On the chromatic number of integral circulant graphs.
\newblock {\em Comput. Math. Appl.}, 60(1):144--150, 2010.

\bibitem{Klotz}
W.~Klotz and T.~Sander.
\newblock Uniquely colorable {C}ayley graphs.
\newblock {\em Ars Math. Contemp.}, 12(1):155--165, 2017.

\bibitem{Konstantinova}
E.~Konstantinova.
\newblock Chromatic properties of the pancake graphs.
\newblock {\em Discuss. Math. Graph Theory}, 37(3):777--787, 2017.

\bibitem{MR0439682}
H.~Meyniel.
\newblock On the perfect graph conjecture.
\newblock {\em Discrete Math.}, 16(4):339--342, 1976.

\bibitem{Mrazovic}
R.~Mrazovi\'c.
\newblock One-point concentration of the clique and chromatic numbers of the
  random {C}ayley graph on {$\Bbb F_2^n$}.
\newblock {\em SIAM J. Discrete Math.}, 31(1):143--154, 2017.

\bibitem{MR1254158}
M.~D. Plummer.
\newblock Well-covered graphs: a survey.
\newblock {\em Quaestiones Math.}, 16(3):253--287, 1993.

\bibitem{MR2496915}
Y.~Wu, W.~Zang, and C.-Q. Zhang.
\newblock A characterization of almost {CIS} graphs.
\newblock {\em SIAM J. Discrete Math.}, 23(2):749--753, 2009.

\bibitem{MR1715546}
M.~Yamashita and T.~Kameda.
\newblock Modeling {$k$}-coteries by well-covered graphs.
\newblock {\em Networks}, 34(3):221--228, 1999.

\bibitem{MR1344757}
W.~Zang.
\newblock Generalizations of {G}rillet's theorem on maximal stable sets and
  maximal cliques in graphs.
\newblock {\em Discrete Math.}, 143(1-3):259--268, 1995.

\end{thebibliography}

\end{document}